\documentclass[a4paper,11pt,reqno]{amsart}
\usepackage{amsthm}
\usepackage{amsmath}
\usepackage{enumerate}
\usepackage{amsfonts}
\usepackage{amssymb}
\usepackage{fullpage}
\usepackage{amsmath,amscd}
\usepackage{stmaryrd}
\usepackage{graphicx}
  \usepackage{gfsartemisia}
  \usepackage{tikz-cd}
\usepackage{array}
\usepackage{amsmath,amssymb,amsfonts,dsfont}
\usepackage[utf8]{inputenc}
\usepackage[english]{babel}
\usepackage[T1]{fontenc}
\usepackage{graphicx}
\usepackage{float}
\usepackage{pdfpages}
\usepackage[a4paper, margin = 3cm, bottom = 3cm]{geometry}
\usepackage{ifpdf}
\usepackage{marginnote}
\usepackage{hyperref}

\usetikzlibrary{calc}
\usetikzlibrary{decorations.pathreplacing,decorations.markings,decorations.pathmorphing}
\usetikzlibrary{positioning,arrows,patterns}
\usetikzlibrary{cd}
\usetikzlibrary{intersections}
\usetikzlibrary{arrows}

\hypersetup{pdfborder=0 0 0,
	    colorlinks=true,
	    citecolor=black,
	    linkcolor=blue,
	    urlcolor=red,
	    pdfauthor={Guillaume Tahar}
	   }

\DeclareMathOperator{\SO}{SO}
\DeclareMathOperator{\PSL}{PSL}

\DeclareMathOperator{\flatsu}{flat}
\DeclareMathOperator{\GL}{GL}
\newtheorem{thm}{Theorem}[section]

\newtheorem{prop}[thm]{Proposition}
\newtheorem{lem}[thm]{Lemma}

\theoremstyle{definition}
\newtheorem{defn}[thm]{Definition}
\theoremstyle{remark}
\newtheorem{rem}[thm]{Remark}
\theoremstyle{definition}

\theoremstyle{definition}
\newtheorem{ex}[thm]{Example}
\theoremstyle{definition}

\numberwithin{equation}{section}
\title{Dihedral monodromy of cone spherical metrics}

\author{Quentin Gendron}
\address[Quentin Gendron]{Instituto de Matem\'{a}ticas de la UNAM
Ciudad Universitaria, CDMX, 04510,
M\'{e}xico}
\email{quentin.gendron@im.unam.mx}

\author{Guillaume Tahar}
\address[Guillaume Tahar]{Beijing Institute of Mathematical Sciences and Applications, Huairou District, Beijing, China}
\email{guillaume.tahar@bimsa.cn}

\date{\today}

\begin{document}

\begin{abstract}
Among metrics of constant positive curvature on a punctured compact Riemann surface with conical singularities at the punctures, dihedral monodromy means that the action of the monodromy group $\mathcal{M} \subset \SO(3)$ globally preserves a pair of antipodal points. Using recent results about local invariants of quadratic differentials, we give a complete characterization of the set of conical angles realized by some cone spherical metric with dihedral monodromy.
\end{abstract}
\maketitle
\setcounter{tocdepth}{1}
\tableofcontents

\section{Introduction}

For a compact Riemann surface $X$ of genus $g$, a finite set of points $A_{1},\dots,A_{n} \in X$ and an array of angles $2\pi(a_{1},\dots,a_{n})$, a natural generalization of the uniformization problem is about the existence of a metric of constant curvature in $X \setminus \lbrace A_{1},\dots,A_{n} \rbrace$ that extends to the points $A_{1},\dots,A_{n}$ as conical singularities of prescribed angles.\newline

In this problem, Gauss-Bonnet theorem states that the sum of diffuse and singular curvature is a topological invariant. The total diffuse curvature is $2\pi(
2-2g-n+\sum a_{i})$. Depending on the sign of this quantity, the metric should be hyperbolic, flat or spherical in the complement of the singularities.\newline
If the latter quantity is nonpositive, the existence of a hyperbolic or a flat metric has been obtained by several authors, see \cite{He, Mc, Tr,Tr1} for references. On the opposite, when this quantity is positive, the existence of a spherical metric is far from being granted.\newline
Some authors have speculated an equivalence between the existence of a metric of constant positive curvature and a notion of stability of bundles (in the spirit of Yau-Tian-Donaldson conjecture for Fano manifolds), see \cite{SLX}.\newline

We could gain a better understanding of cone spherical metrics by the study of metrics with constrained monodromy. In spherical metrics with \textit{co-axial monodromy}, the monodromy is contained in the group of rotations around a given axis. The realization problem has been solved by Mondello and Panov in \cite{MP} for the class of metrics with non co-axial monodromy. In the case of metrics with co-axial monodromy, it has been solved by Eremenko in \cite{Er}.\newline

A cone spherical metric has \textit{dihedral monodromy} if its monodromy group (as a subgroup of $\SO(3)$) globally preserves a pair of antipodal points (see~\cite{SCLX}). This is equivalent to the global preservation of the dual great circle. The rotations of the monodromy group are rotations around the axis and rotations of angle $\pi$ around any axis whose antipodal points belong to the preserved great circle (see Section~\ref{sec:SO3}). The metrics with co-axial monodromy form a subclass of the metrics with dihedral monodromy.\newline

In \cite{SCLX}, the authors obtain the developing map of a spherical structure with dihedral monodromy by the integration of a square root of a quadratic differential. In this way, they prove existence of cone spherical metrics with dihedral monodromy for some arrays of angles. The technical result in their paper (Theorem~1.8) amounts to a characterization of the arrays of residues at the poles a quadratic differential may have. They considered only the case of differentials with simple zeroes and double poles.\newline
Their work can be extended and completed. Indeed, in the recent papers \cite{GT,GT1} is given a complete characterization of the local invariants an Abelian or a quadratic differential can realize. Using the solution of this problem (Question 1.7 in \cite{SCLX}), we are able to give an explicit characterization of the arrays of angles that can be realized by a cone spherical metric with dihedral monodromy. The key issue, explained in Section~\ref{sec:JSdif}, is the connection between a subclass of spherical metrics (hemispherical surfaces introduced in Definition~\ref{def:hemsurf}) and a subclass of quadratic differentials (totally real Jenkins-Strebel differentials introduced in Definition~\ref{def:totrealJS}).\newline

\subsection{Main results}

In a cone spherical metric realizing some array of angles, we distinguish between:
\begin{itemize}
\item \textit{even} conical singularities for which the angle is in $2\pi\mathbb{Z}$;
\item \textit{odd} conical singularities for which the angle is in $\pi(2\mathbb{Z}+1)$;
\item \textit{non-integer} conical singularities for which the angle is not in $\pi\mathbb{Z}$.
\end{itemize}

For a metric with $n$ conical singularities, we have $n=n_{E}+n_{O}+n_{N}$. These three terms are respectively the numbers of even, odd and non-integer conical singularities.

\begin{defn}\label{def:41}
We consider arrays of angles $2\pi(a_{1},\dots,a_{n_{E}},b_{1},\dots,b_{n_{O}},c_{1},\dots,c_{n_{N}})$ where the three subfamilies are respectively even ($a_{i}\in\mathbb{Z}$), odd ($b_{i}\in\mathbb{Z}+\frac{1}{2}$) and non-integer angles ($c_{i}\notin\frac{1}{2}\mathbb{Z}$).\newline
The \textit{total sum} is defined by $\sigma=\sum a_{i} + \sum b_{j} + \sum c_{k}$.\newline
The \textit{maximal integral sum} is defined by:
\begin{itemize}
    \item $T=\sum a_{i} + \sum\limits_{j=2}^{n_{O}} b_{j}$ if $n_{O}$ is odd (assuming  $b_{1}\leq b_{2}\leq \dots\leq b_{n_{O}}$);
    \item $T=\sum a_{i} + \sum b_{j}$ if $n_{O}$ is even.
\end{itemize}
In particular, $T \in \mathbb{Z}$ and $T \leq \sigma$.
\end{defn}

Just like Gauss-Bonnet inequality $\sigma > 2g+n-2$ is a necessary condition for the existence of a cone spherical metric with prescribed angles at the conical singularities, the main necessary condition for an array of angles to be realized by a spherical metric with dihedral monodromy is the \textit{strengthened Gauss-Bonnet inequality}.

\begin{thm}\label{thm:GBplus}
Let $2\pi(a_{1},\dots,a_{n_{E}},b_{1},\dots,b_{n_{O}},c_{1},\dots,c_{n_{N}})$ be an array of angles. If this array of angles is realized by a cone spherical metric with dihedral monodromy on a surface of genus~$g$, then it should satisfy the \textit{strengthened Gauss-Bonnet inequalities}:
\begin{itemize}
\item $T \geq 2g+n-1$ if $n_{O}$ is even and $n_{N}=0$;
\item $T \geq 2g+n-2$ otherwise.
\end{itemize}
\end{thm}

We briefly explain why this condition appears. A conical singularity of the spherical metric either corresponds to a double pole of the quadratic differential whose residue is determined by the conical angle or a point of order $k \geq -1$ of the differential. The strengthened Gauss-Bonnet inequality then follows from the fact that the sum of the orders of the singularities of a quadratic differential on a genus $g$ Riemann surface is $4g-4$.\newline

In the general case, the strenghtened Gauss-Bonnet inequality is also a sufficient condition for the existence of cone spherical metrics with dihedral monodromy. However, there are  exceptional families of arrays of angles that cannot be realized by a spherical metric with dihedral monodromy. They come from general arithmetic obstructions to the existence of quadratic differentials with integer residues obtained in \cite{GT,GT1}. The list of obstructions and the complete characterization of arrays of angles that can be realized are contained in six theorems.\newline
For spherical metrics in genus zero, the classification is given in Theorems~\ref{thm:63},~\ref{thm:65} and~\ref{thm:610} for the strict dihedral case. We remind of \cite[Theorem 1]{Er} for the co-axial case in Theorem~\ref{thm:61}.\newline
For spherical metrics in higher genus, the classification is given in Theorem~\ref{thm:52} for the strict dihedral case and Theorem~\ref{thm:53} for the co-axial case.\newline

The organization of the paper is the following:
\begin{itemize}
\item In Section~\ref{sec:dihedral}, we introduce the co-axial and dihedral monodromy classes. We draw the connection between quadratic differentials and cone spherical metrics with dihedral monodromy.
\item In Section~\ref{sec:difprescrites}, we present the results on quadratic and Abelian differentials with prescribed orders of zeroes and poles.
\item In Section~\ref{sec:GBplus}, we state the strenghtened Gauss-Bonnet inequality which is the main necessary condition for realization of an array of angles by a cone spherical metric with dihedral monodromy.
\item In Section~\ref{sec:highergenus}, we give a characterization of arrays of angles that can be realized in a spherical surface of genus $g \geq 1$ with co-axial or dihedral monodromy (Theorems~\ref{thm:52} and~\ref{thm:53}), comparing the two classes.
\item In Section~\ref{sec:genuszero}, we state the theorem of Eremenko which characterizes the arrays of angles for cone spherical metrics with co-axial monodromy in genus zero. Comparatively, we state and prove the analogous result for metrics with dihedral monodromy (Theorems~\ref{thm:63},~\ref{thm:65} and~\ref{thm:610}).
\end{itemize}

\section{From differentials to spherical metrics with dihedral monodromy}
\label{sec:dihedral}

In this section, we recall the basic statements about a special class of spherical surfaces. Then we show how they are related to quadratic differentials and give some background on them.

\subsection{Spherical structures}
\label{sec:SO3} On a compact surface $S$ with a finite set $\left\{A_{1},\dots,A_{n}\right\}$  of singular points, a \textit{spherical structure} is an atlas of charts on $S \setminus \left\{A_{1},\dots,A_{n}\right\}$ with values in the standard sphere $\mathbb{S}^{2}$ whose transition maps belong to $\SO(3)$.\newline

A singular point of a spherical surface is a \textit{conical singularity} of angle $\theta$ if it is locally isometric to the singularity of a hemispherical sector of angle $\theta$ (see Definition~\ref{defn:HSector}).

\begin{defn}\label{defn:HSector}
A \textit{hemispherical sector} of angle $\alpha \in ]0;2\pi]$ is a singular surface with boundary, obtained by considering the sector angle $\alpha$ between two meridians in the standard half-sphere and identifying these two sides. It has a conical singularity of angle $\alpha$ and a geodesic boundary of length $\alpha$.
\end{defn}

The definition of hemispherical sectors extends by ramified cover to any angle $\alpha \in \mathbb{R}_{+}^{\ast}$. They provide a geometric model for every conical singularity.\newline

Given a spherical structure on a surface $S$, there is a representation of the fundamental group of $S \setminus \left\{A_{1},\dots,A_{n}\right\}$ into the group of linear-fractional transformations. Its image is the {\em monodromy} of the spherical structure. Note that it is a subgroup of $\SO(3)$. A first natural subclass of spherical metrics is given by
metrics with \textit{reducible} or \textit{co-axial} monodromy.\newline

\begin{defn}
A cone spherical metric has \textit{co-axial monodromy} if its monodromy group fixes each element of an antipodal pair of points (an axis).
\end{defn}

In this paper, we focus on a class of metrics with slightly more general monodromy, introduced in~\cite{SCLX}.

\begin{defn}
A cone spherical metric has \textit{dihedral monodromy} if its monodromy group preserves a unordered pair of antipodal points.
\end{defn}

Among cone spherical metrics with dihedral monodromy, we distinguish co-axial monodromy (preserving pointwise the two antipodal points of the axis) and strict dihedral monodromy (dihedral but not co-axial).\newline

\subsection{Latitude foliation}
\label{sec:latitude}

The dihedral monodromy preserves a great circle in the sphere. We will refer to it as the \textit{equator}. We will also refer to the \textit{equatorial net} as the locus in the spherical surface that is mapped to the equator in every chart. Since monodromy acts by isometries, it also preserves the {\em latitude foliation} that decomposes the sphere into circles of constant latitude. Besides, the \textit{absolute latitude}, which  is the absolute value of the angular distance of a point of the sphere to the equator, is also well-defined.\newline

We rephrase the previous paragraph. Given a surface $S$ with a cone spherical metric with dihedral monodromy, the \textit{absolute latitude} is the map $\phi\colon S \rightarrow [0,\frac{\pi}{2}]$ which associates to a point the norm of its latitude. The preimage of $0$ by $\phi$ is the \textit{equatorial net} while the preimage of $\frac{\pi}{2}$ is the \textit{polar locus}. As we will see, if the monodromy is in fact co-axial, then the sign of the latitude is globally defined.\newline

\begin{rem}
The only circle of latitude which is a geodesic is the equator. The others are just loxodromic paths.
\end{rem}

The local geometry of conical singularities induces constraints on their positions in the latitude foliation. A conical singularity $A$ of angle $\alpha$ should satisfy the following conditions:
\begin{itemize}
\item if $\alpha \notin \pi\mathbb{Z}$, then $\phi(A)=\frac{\pi}{2}$;
\item if $\alpha \notin 2\pi\mathbb{Z}$, then $\phi(A)\in \lbrace 0;\frac{\pi}{2} \rbrace$;
\end{itemize}

In other words, a point of the polar locus can support a conical singularity of any angle. Away from the polar locus, the angle of the singularity should be an integer multiple of $\pi$ corresponding to the number of distinct branches of the foliation that approach the singularity. If this number is odd, then the singularity automatically belongs to the equatorial net. Indeed, the latitude circle the singularity belongs to should be preserved by the nontrivial monodromy of a simple loop around this singularity.\newline

\begin{rem}
For the metrics with co-axial monodromy, the \textit{latitude} (not just the absolute latitude) defined in each chart is preserved by the monodromy and it thus globally defined. This implies in particular that every singularity whose angle does not belong to $2\pi\mathbb{Z}$ should belong to the polar locus of the latitude foliation.
\end{rem}

\subsection{Hemispherical surfaces}
\label{sec:hemispherical}

\begin{defn}\label{def:hemsurf}
A \textit{hemispherical surface} is a closed surface with a cone spherical metric obtained by gluing finitely many hemispherical sectors along their geodesic boundary. The gluing can identify several boundary points and create conical singularities whose angles are integer multiples of~$\pi$.
\end{defn}

An example of hemispherical surface obtained by gluing three hemispherical sectors is shown in the left of Figure~\ref{fig:corrsphflat}. In that figure, the letters indicate that the boundary segments are glued together by a rotation. This surface has four singularities, the one lying in the equator being of angle $4\pi$.

This class of surfaces is an example of spherical surface with dihedral monodromy.

\begin{prop}
A hemispherical surface has dihedral monodromy.
\end{prop}

\begin{proof}
For a spherical surface obtained by the gluing of hemispherical sectors, a latitude foliation is defined in each sector and extends globally through the boundary. The boundary of each sector belongs to the equatorial net so the absolute latitude function is well defined on the whole surface. Consequently, the spherical structure has dihedral monodromy.
\end{proof}

A specific property of hemispherical surfaces is that the latitude of every singularity is either $0$ or $\frac{\pi}{2}$.\newline

Hemispherical surfaces are really simple to describe. They are completely characterized by the lengths of the boundary segments and the combinatorics of the gluing pattern. The equatorial net of a hemispherical surface is the union of the boundaries of every cylinder. We show in Proposition~\ref{prop:isomonodromy} that any spherical surface with dihedral monodromy can be deformed to a hemispherical surface.

\begin{ex}
The most basic hemispherical surface is obtained from a hemispherical sector of angle $\alpha$. We divide the boundary into two geodesic segments of length $\frac{\alpha}{2}$ and glue them on each other. We obtain a spherical surface with two conical singularities of angle $\pi$ on the equator and one singularity of angle $\alpha$ at the pole.
\end{ex}

\subsection{Quadratic differentials and half-translation surfaces}
\label{sec:quaddifhalf}

On a Riemann surface $X$, a quadratic differential $q$ is a meromorphic section of $K^{\otimes 2}_{X}$. Outside of the zeroes and poles of odd order, the quadratic differential $q$ is locally the square of a meromorphic $1$-form $\pm \sqrt{q}$.\newline
The antiderivatives of $\pm \sqrt{q}$ are the developing maps of a \textit{half-translation structure}. This structure is formed by an atlas in the complement of the singularities of $q$. This atlas is made of charts to $\mathbb{C}$ whose transition maps are of the form $z \mapsto \pm z +c$.\newline
For a quadratic differential that is locally given by $f(z)dz^{2}$, the associated flat metric is $|\sqrt{f}|.|dz|$.\newline

A singularity of order $a \geq -1$ corresponds to a conical singularity of angle $(2+a)\pi$ in the associated flat metric. A double pole is a point at infinity in an infinite cylinder. Besides, at a double pole, if $q=(\frac{r_{-2}}{z^{2}}+\frac{r_{-1}}{z}+r_{0}+r_{1}z+\dots)dz^{2}$, then the quadratic residue at the double pole is $r_{-2}$. The residue $r_{-2}$ should always be nonzero. In the cylinder neighboring the double pole of quadratic residue $r$, the flat cylinder is obtained by identifying the two sides of an half-infinite band by a translation of $\pm 2\pi\sqrt{r}$, see \cite{St} for details.\newline

We define a special class of quadratic differentials which is illustrated on the right of Figure~\ref{fig:corrsphflat}.

\begin{defn}\label{def:totrealJS}
A meromorphic quadratic differential $q$ on a Riemann surface $X$ is a {\em totally real Jenkins-Strebel differential} if its associated half-translation surface is formed by gluing the boundary of finitely many semi-infinite cylinders along horizontal segments.
\end{defn} 

Recall that a {\em period} of $q$ is the integral of $\pm\sqrt{q}$ along a path between two singularities of order $\geq-1$. A period is {\em absolute} if both starting and ending points coincide. The set of absolute periods forms a subgroup of $\mathbb{C}$. Note that the periods of a totally real Jenkins-Strebel differential are real. It has exactly one double pole for each cylinder and the other singularities are conical singularities of angle in $\pi\mathbb{Z}$. They are either simple poles or zeroes of the differential.\newline

\subsection{Relation between hemispherical surfaces and totally real Jenkins-Strebel differentials}
\label{sec:JSdif}

There is a construction (introduced by Eremenko in \cite{Er}) which associates to any hemispherical surface $S$ a flat surface $S_{\flatsu}$. It replaces every hemispherical sector of angle $2a\pi$ by a semi-infinite cylinder of circumference~$2a\pi$, see Figure~\ref{fig:corrsphflat}. The lengths of the segments in the boundary are preserved.\newline

Conversely, given a half-translation surface defined by a totally real Jenkins-Strebel differential, we replace each semi-infinite cylinder by a hemispherical sector. This operation is the inverse of the previous one.\newline

\begin{figure}[htb]
\begin{tikzpicture}[scale=1.1]

\begin{scope}[xshift=-6cm,yshift=.75cm]
 \coordinate (a) at (0,0);
\coordinate (b) at (1.25,0);
\coordinate (c) at (2,0);
\coordinate (d) at (.625,1);
\coordinate (e) at (1.625,1);
\coordinate (f) at (1,-1);
  \filldraw[fill=black!10] (a)  .. controls ++(90:.5) and ++(-150:.5) .. (d) coordinate[pos=.5](g)  .. controls ++(-30:.5) and ++(90:.5) .. (b) coordinate[pos=.5](h) .. controls ++(90:.5) and ++(-130:.5) .. (e) coordinate[pos=.5](i)  .. controls ++(-50:.5) and ++(90:.5) ..  (c) coordinate[pos=.5](j) .. controls ++(-90:.5) and ++(10:.5) .. (f) coordinate[pos=.5](k) .. controls ++(170:.5) and ++(-90:.5) .. (a) coordinate[pos=.5](l);

\draw (a) -- (b)coordinate[pos=.5](m) -- (c)coordinate[pos=.5](n);
\node at (m) {$1$};
\node at (n) {$2$};
\node at (g) {$\alpha$};\node at (h) {$\alpha$};
\node at (i) {$\beta$};\node at (j) {$\beta$};
\node at (k) {$\gamma$};\node at (l) {$\gamma$};

\draw (a)--  (b) coordinate[pos=.5](gm) -- (c) coordinate[pos=.5](n);
    \fill (a)  circle (1pt);
\fill[] (b) circle (1pt);
    \fill (c)  circle (1pt);
    \filldraw[fill=white] (d) circle (1pt);
    \filldraw[fill=blue] (e) circle (1pt);
        \filldraw[fill=red] (f) circle (1pt);
% \draw (a)-- (d);
% \draw (b)-- (d);
% \draw (b)-- (e);
% \draw (c)-- (e);
% \draw (a)-- (f);
% \draw (c)-- (f);
\end{scope}

%surface plate
\begin{scope}[xshift=-.65cm]
\coordinate (a) at (-1,1);
\coordinate (b) at (.25,1);

    \fill[fill=black!10] (a)  -- (b)coordinate[pos=.5](f) -- ++(0,1.2) --++(-1.25,0) -- cycle;
    \fill (a)  circle (1pt);
\fill[] (b) circle (1pt);
 \draw  (a) -- (b);
 \draw (a) -- node{$a$}++(0,1.1) coordinate (d)coordinate[pos=.5](h);
 \draw (b) --node{$a$} ++(0,1.1) coordinate (e)coordinate[pos=.5](i);
 \draw[dotted] (d) -- ++(0,.2);
 \draw[dotted] (e) -- ++(0,.2);
\node at (f) {$1$};
\end{scope}
\begin{scope}[xshift=.65cm]
\coordinate (a) at (-.75,1);
\coordinate (b) at (0,1);

    \fill[fill=black!10] (a)  -- (b)coordinate[pos=.5](f) -- ++(0,1.2) --++(-.75,0) -- cycle;
    \fill (a)  circle (1pt);
\fill[] (b) circle (1pt);
 \draw  (a) -- (b);
 \draw (a) --node{$b$} ++(0,1.1) coordinate (d)coordinate[pos=.5](h);
 \draw (b) --node{$b$} ++(0,1.1) coordinate (e)coordinate[pos=.5](i);
 \draw[dotted] (d) -- ++(0,.2);
 \draw[dotted] (e) -- ++(0,.2);
\node at (f) {$2$};
\end{scope}

\begin{scope}[xshift=-.5cm,yshift=-.5cm]
\coordinate (a) at (-1,1);
\coordinate (b) at (1,1);
\coordinate (c) at (.25,1);

    \fill[fill=black!10] (a)  -- (c)coordinate[pos=.5](f)-- (b)coordinate[pos=.5](g) -- ++(0,-1.2) --++(-2,0) -- cycle;
    \fill (a)  circle (1pt);
\fill[] (b) circle (1pt);
    \fill (c)  circle (1pt);
 \draw  (a) -- (b);
 \draw (a) --node{$c$} ++(0,-1.1) coordinate (d)coordinate[pos=.5](h);
 \draw (b) --node{$c$} ++(0,-1.1) coordinate (e)coordinate[pos=.5](i);
 \draw[dotted] (d) -- ++(0,-.2);
 \draw[dotted] (e) -- ++(0,-.2);
\node at (f) {$1$};
\node at (g) {$2$};
\end{scope}

\end{tikzpicture}
\caption{The correspondence between a hemispherical surface and a flat surface associated to a totally real Jenkins-Strebel differential.} \label{fig:corrsphflat}
\end{figure}
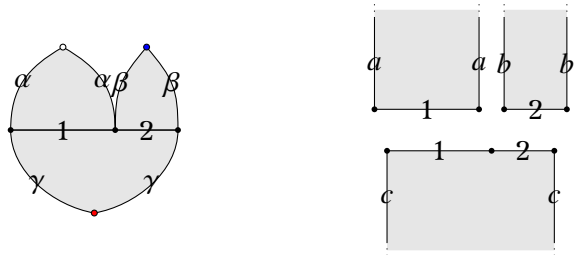

The group of absolute periods of the quadratic differential is the image of the monodromy group of the hemispherical surface into the rotation group along the preserved axis (up to a factor $2\pi$).\newline

A totally real Jenkins-Strebel differential $q$ is the global square of $1$-form if and only if every cylinder of the flat surface can be given a sign such that every boundary segment bounds a positive and a negative cylinder.\newline
The latter condition implies in particular that singularities should be of even order (with an angle in $2\pi\mathbb{Z}$).\newline
Translated in the language of spherical geometry, a hemispherical surface has globally defined latitude if and only if its totally real Jenkins-Strebel differential is the global square of a $1$-form. This property is equivalent to co-axial monodromy. The surfaces pictured in Figure~\ref{fig:corrsphflat} are examples of such behaviour.\newline

\subsection{Some background on complex projective structures and quadratic differentials}
\label{sec:backgroundquad}

In this section, we recall some facts about complex projective structures and their relation with quadratic differentials. This gives interesting background on this work but do not play a major role in the remaining sections.\newline

A {\em complex projective structure} is defined by charts with values in $\mathbb{CP}^{1}$ and transition maps in $\PSL(2,\mathbb{C})$. In particular, a spherical structure is a complex projective structure whose monodromy is conjugated to a subgroup of $\SO(3)$.\newline

A complex projective structure has {\em dihedral monodromy} if it globally preserves a pair of points of $\mathbb{CP}^{1}$, see Section~6 in \cite{FG}. Up to conjugation, we can assume these preserved points are $\lbrace{ 0 , \infty \rbrace}$. Then the monodromy will act by functions of the form $z \mapsto az^{\pm 1}$ with $a \in \mathbb{C}^{\ast}$.\newline
This directly implies existence of a quadratic differential $q$ such that any developing map of the structure is of the form $e^{\int^{z} \sqrt{-q}}$. Besides, in order to avoid irregular singularities of the developing map, it is reasonable to restrict to quadratic differentials with at worst double poles (logarithmic singularities).\newline

In this class of complex projective structures, the monodromy preserves a metric if and only if all multiplicative factors appearing in the transition functions have modulus $1$. In other words, periods of the quadratic differential should be real (see Section~1.1 of \cite{SCLX} for details). In particular, the quadratic differential is Jenkins-Strebel (see \cite{St}).\newline

Then, the totally real Jenkins-Strebel quadratic differential introduced in Definition~\ref{def:totrealJS} are those for which the zeroes and poles either belong to the equator or the poles of the latitude foliation defined by the corresponding core spherical metric. The whole  correspondence is summarised in Table~\ref{table:cor}.\newline

\begin{table}[h]
\begin{tabular}{|c|c|l|r|} 
\hline  
Complex-analytic object & Geometric structure\\  
\hline \hline
Quadratic differential & Complex projective structure\\
& with dihedral monodromy \\  
\hline 
Jenkins-Strebel quadratic differential & Cone spherical metric \\
& with dihedral monodromy\\ 
\hline 
Totally real Jenkins-Strebel differential & Hemispherical surface \\
\hline
\end{tabular}
\caption{Correspondence between analytic and geometric objects.}
\label{table:cor}
\end{table}

\section{Differentials with prescribed singularities}
\label{sec:difprescrites}

In this section, we recall some results of \cite{GT,GT1} about existence of differentials with prescribed local behaviour. Moreover, we recall an important operation on differentials called the {\em contraction flow}.

\subsection{Quadratic differentials}
\label{sec:quaddif}

By abuse of notation, we refer to a pair $(X,q)$ where $X$ is a Riemann surface of genus $g$ and $q$ a quadratic differential on $X$. The moduli spaces of these objects are stratified according to the orders of the singularities of~$q$. In this paper, quadratic differentials have at worst double poles. For $a_{1},\dots,a_{n} \in 2\mathbb{N}^{\ast}$ and $b_{1},\dots,b_{m} \in 2\mathbb{N}-1$,  the {\em stratum of quadratic differentials} with zeroes and simple poles of multiplicities $a_{1},\dots,a_{n},b_{1},\dots,b_{m}$ and $p$ double poles is denoted by $\mathcal{Q}(a_{1},\dots,a_{n},b_{1},\dots,b_{m},-2^{p})$. See \cite{La} for details. \newline

We also require that quadratic differentials are not global squares of $1$-forms. Such differentials are called \textit{primitive}. Riemann-Roch theorem (or Gauss-Bonnet theorem) implies that $\sum a_{i}+\sum b_{j}-2p=4g-4$.\newline

For such quadratic differentials, classical complex analysis shows that at a point of order $m \geq -2$, the differential is locally biholomorphic to
\begin{itemize}
    \item $z^{m}dz^{2}$ if $m \geq -1$;
    \item $\left(\frac{rdz}{z}\right)^{2}$ if $m=-2$.
 \end{itemize}   
In the latter case, the {\em quadratic residue} $r^{2}$ is a local invariant (see \cite{St} for details).\newline

Theorems~1.1,~1.2,~1.3 and~1.9 of~\cite{GT1} provide a complete characterization of the arrays of quadratic residues that can be realized in any given stratum of quadratic differentials. We extract from these results the following two theorems.

\begin{thm}[\cite{GT1}]\label{thm:31}
Let $\mathcal{Q}$ be a stratum $\mathcal{Q}(a_{1},\dots,a_{n},b_{1},\dots,b_{m},-2^{p})$ of primitive quadratic differentials on a surface of genus $g \geq 1$ such that $p \geq 1$.\newline
Every configuration of quadratic residues $(r_{1},\dots,r_{p})\in(\mathbb{R}_{+}^{\ast})^{p}$ is realized by a quadratic differential of $\mathcal{Q}$ with the exception of configurations $(r,\dots,r)$ for the strata $\mathcal{Q}(4s,-2^{2s})$ and $\mathcal{Q}(2s+1,2s-1,-2^{2s})$ in genus one.\newline
\end{thm}

In genus zero, it should be noted that a quadratic differential is primitive if and only if at least one of its singularities is of odd order. 

\begin{thm}[\cite{GT1}]\label{thm:32}
Let $\mathcal{Q}$ be a stratum $\mathcal{Q}(a_{1},\dots,a_{n},b_{1},\dots,b_{m},-2^{p})$ of primitive quadratic differentials on the Riemann sphere with $m,p \geq 1$.\newline
Every configuration of quadratic residues $(r_{1},\dots,r_{p})\in(\mathbb{R}_{+}^{\ast})^{p}$ is realized by a quadratic differential of $\mathcal{Q}$ with the following exceptions:
\begin{itemize}
\item $\mathcal{Q}(p-2,p-2,-2^{p})$ with $p$ odd and configurations of the form $(A^{2},B^{2},C^{2},\dots,C^{2})$ with $C=A+B$ or $B=A+C$ and $A,B,C>0$;
\item $\mathcal{Q}(p-1,p-3,-2^{p})$ with $p$ even and configurations of the form $(A^{2},A^{2},B^{2},\dots,B^{2})$ with $A,B>0$;
\item $\mathcal{Q}(a_{1},\dots,a_{n},b_{1},b_{2},-2^{p})$ and configurations of the form $(L \cdot f_{1}^{2},\dots,L \cdot f_{p}^{2})$ with $L>0$, $f_{1},\dots,f_{p} \in \mathbb{N}^{\ast}$, $\sum f_{j}$ is even and $\sum f_{j} \leq b_{1}+b_{2}+2$;
\item $\mathcal{Q}(a_{1},\dots,a_{n},b_{1},b_{2},-2^{p})$ and configurations of the form $(L \cdot f_{1}^{2},\dots,L \cdot f_{p}^{2})$ with $L>0$, $f_{1},\dots,f_{p} \in \mathbb{N}^{\ast}$, $\sum f_{j}$ is odd and $\sum f_{j} \leq \max(b_{1},b_{2})$.\newline
\end{itemize}
\end{thm}

\subsection{Abelian differentials}
\label{sec:difabel}

A connected component of a stratum of quadratic differentials is either entirely made of squares of Abelian differentials or entirely made of non-squares. In this Section, we describe obstructions to the realization of configurations of residues by quadratic differentials that are global squares of $1$-forms.\newline

We consider {\em stratum} $\mathcal{H}(a_{1},\dots,a_{n},-1^{p})$ of pairs $(X,\omega)$ where $X$ is a compact Riemann of genus $g$ and $\omega$ is a meromorphic $1$-form with zeroes of orders $a_{1},\dots,a_{n}$ and $p$ simple poles. Gauss-Bonnet theorem implies that $\sum a_{i}-p=2g-2$.\newline
We consider configurations of real numbers $(\lambda_{1},\dots,\lambda_{x},-\mu_{1},\dots,-\mu_{y})$ with $x+y=p$, whose sum is zero (to satisfy residue theorem) and such that $\lambda_{1},\dots,\lambda_{x},\mu_{1},\dots,\mu_{y}>0$.\newline

In \cite{GT}, we show in Theorem~1.1 that in genus at least one, the only obstruction is the Residue theorem.

\begin{thm}[\cite{GT}]\label{thm:33}
In any stratum $\mathcal{H}(a_{1},\dots,a_{n},-1^{p})$ of meromorphic $1$-forms on a surface of genus $g \geq 1$ with $p \geq 1$, every configuration of residues $(\lambda_{1},\dots,\lambda_{x},-\mu_{1},\dots,-\mu_{y})$ with $\sum \lambda_{i} = \sum \mu_{j}$ is realized by a differential in the stratum.
\end{thm}

In genus zero, there is an additional arithmetic obstruction, described in Theorem~1.2 of~\cite{GT} and Theorem~2 of~\cite{Er}.

\begin{thm}\cite{GT}\label{thm:34}
In any stratum $\mathcal{H}(a_{1},\dots,a_{n},-1^{p})$ of meromorphic $1$-forms on a surface of genus zero with $p \geq 1$, every configuration of residues $(\lambda_{1},\dots,\lambda_{x},-\mu_{1},\dots,-\mu_{y})$ with $\sum \lambda_{i} = \sum \mu_{j}$ is realized by a differential in the stratum with the following exception.\newline
If the configuration of residues is of the form $L( f_{1},\dots,f_{x},-g_{1},\dots,-g_{y})$ with $L>0$ and $f_{1},\dots,f_{x},g_{1},\dots,g_{y}$ are integers without nontrivial common factor, then it can be realized in the stratum if and only if $\sum f_{i}= \sum g_{j} > \max(a_{1},\dots,a_{n})$.\newline
\end{thm}

\subsection{Contraction flow}
\label{sec:flotcontr}

In Sections~\ref{sec:quaddif} and~\ref{sec:difabel}, we considered quadratic differentials defining flat surfaces of infinite area (because cylinders around double poles are semi-infinite). There is an action of $\GL^{+}(2,\mathbb{R})$ on strata of quadratic differentials. This group acts by postcomposition in the charts and acts naturally on the periods, see \cite{Zo} for details.\newline

The \textit{contraction flow} is a one-parameter flow in $\GL^{+}(2,\mathbb{R})$ that preserves a direction and contracts exponentially another. For a quadratic differential $q$ defining a flat surface of infinite area, we consider the contraction flow contracting a generic direction (a direction where there is no saddle connection). The trajectory of $q$ along the flow has a limit inside the same stratum. For the limit differential, every (absolute and relative) period belongs to the preserved direction, see Section~5.4 of \cite{Ta}.\newline
Infinite area hypothesis is necessary because otherwise the area of the flat surface would shrink along the flow.\newline

If we apply the contraction flow to any differential that realizes a real configurations of residues (like in Sections~\ref{sec:quaddif} and~\ref{sec:difabel}), a generic contraction flow preserving the horizontal direction will converge to a \textit{totally real Jenkins-Strebel differential}.\newline
According to Section~\ref{sec:backgroundquad}, quadratic differentials with real periods correspond to cone spherical metrics with dihedral monodromy. We use the contraction flow on quadratic differentials to construct an isomonodromic deformation of any such cone spherical metric to hemispherical surface (see Section~\ref{sec:hemispherical}).

\begin{prop}\label{prop:isomonodromy}
Any array of angles realized by a cone spherical metric with dihedral monodromy is realized by a hemispherical surface with the same monodromy.
\end{prop}

\begin{proof}
If an array of angle is realized by a cone spherical metric with dihedral monodromy, then the metric is given by a quadratic differential $q_{0}$ with real periods (see Section~\ref{sec:backgroundquad}). In the half-translation structure defined by $q_{0}$, we consider a non-horizontal direction $\alpha$ that is not the direction of a saddle connection. The action of the contraction flow preserving the horizontal direction and contracting a generic direction defines a path $q_{t}$ inside the stratum containing $q_{0}$. The orders of zeroes and poles of $q_{t}$ remain identical. Besides, since the real direction (equivalently the horizontal direction) is preserved, real residues at the double poles also remain the same.\newline
In Section~5.4 of \cite{Ta} it has been proved that such a contraction flow contracting a generic direction has a limit $q_{\infty}$ inside the ambient stratum. Besides, the direction of every geodesic segment joining the conical singularities of $q_{\infty}$ is horizontal. In other words, $q_{\infty}$ is a totally real Jenkins-Strebel differential. Following Section~\ref{sec:JSdif}, the differential $q_{\infty}$ defines a hemispherical surface.\newline
Along the path $q_{t}$, the periods of the differential along any closed loop have remained identical. In other words, the deformation of the corresponding cone spherical metrics is isomonodromic.
\end{proof}

\section{Strengthened Gauss-Bonnet inequality}
\label{sec:GBplus}

For an array of angles and a given genus $g$, we would like to know if the array of angles is realized by a cone spherical metric with dihedral monodromy on a surface of genus $g$. Conical singularities with integer or half-integer angles play a special role, see Definition~\ref{def:41}.\newline

Just like a Gauss-Bonnet inequality is necessary for realization of an array of angles by a cone spherical metric, a necessary condition for realization by a spherical metric with a dihedral monodromy is that a strengthened Gauss-Bonnet inequality should be fulfilled by the conical singularities with integer and half-integer angles alone.

\begin{proof}[Proof of Theorem~\ref{thm:GBplus}]
Following Section~\ref{sec:hemispherical}, if such an array of angles is realized by a cone spherical metric with dihedral monodromy, then it can be realized by a hemispherical surface $S$. In this hemispherical surface, the singularities with non-integer angles are automatically at the poles of the latitude foliation.\newline
We consider the totally real Jenkins-Strebel differential $q$ corresponding to hemispherical surface $S$. The quadratic differential $q$ belongs to a stratum $\mathcal{Q}(d_{1},\dots,d_{s},-2^{t})$. Here, $d_{1},\dots,d_{s} \geq -1$ are the orders of the zeroes and simple poles of $q$ while $t$ is the number of double poles. Since the total order of the zeroes and poles of a meromorphic differential on a compact Riemann surface of genus $g$ is $4g-4$, we have $d_{1}+\dots+d_{s}=4g-4+2t$. In particular, $d_{1}+\dots+d_{s}$ is an even number\newline
We consider the conical singularities corresponding to the points of orders $d_{1},\dots,d_{s}$ of $q$. Their conical angles belong in $\pi\mathbb{Z}$ and their total angle is $\pi\sum (d_{i}+2)$. The latter sum belongs to $2\pi \mathbb{Z}$. We have thus $2T \geq \sum (d_{i}+2)$ (see Definition~\ref{def:41}). It follows that $2T \geq 4g-4+2t+2s$ and $T \geq 2g+n-2$.\newline
If $n_{O}$ is even, $n_{N}=0$ and $T=2g+n-2$, then $\sigma=T$ and the classical Gauss-Bonnet inequality requires $\sigma>2g+n-2$.
\end{proof}

\section{Spherical structures in higher genus}
\label{sec:highergenus}

We apply the results about quadratic and Abelian differentials with prescribed residues in order to construct hemispherical surfaces realizing the adequate array of angles.\newline

Using equivalence discussed in Section~\ref{sec:JSdif}, we construct hemispherical surfaces corresponding to totally real Jenkins-Strebel differentials. It should be noted that in genus at least one, in strata of quadratic and Abelian differentials, there are several degrees of freedom in addition to the residues. In other words, if there is a differential that realize some configuration of real residues, then there is another differential that realizes it and such that the group of periods of the differential is dense in $\mathbb{R}$. For such a totally real Jenkins-Strebel differential, the monodromy of the corresponding hemispherical surface is dense in the subgroup of $\SO(3)$ preserving a pair of antipodal points (pointwise in the co-axial case and globally in the dihedral case).\newline

\subsection{Strict dihedral case}
\label{sec:stricdihedral}

Specific obstruction in genus one for quadratic differentials with prescribed residues (Theorem~\ref{thm:31}) lead to four exceptional families of arrays of angles that cannot appear in spherical surfaces with strict dihedral monodromy.

\begin{prop}\label{prop:51}
Four exceptional families of arrays of angles cannot be realized by a cone spherical metric with strict dihedral monodromy on a torus:
\begin{itemize}
\item $(4k+2)\pi,c,\dots,c$ with $2k$ non-integer angles $c \notin \pi\mathbb{Z}$;
\item $(2k+3)\pi,(2k+1)\pi,c,\dots,c$ with $2k$ non-integer angles $c \notin \pi\mathbb{Z}$;
\item $(4k+2)\pi$ for any integer $k$;
\item $(2k+3)\pi,(2k+1)\pi$ for any integer $k$.
\end{itemize}
\end{prop}

\begin{proof}
According to Proposition~\ref{prop:isomonodromy}, if such an array of angles is realized by a cone spherical metric with dihedral monodromy, then it can be done by a hemispherical surface. In this hemispherical surface, singularities with non-integer angles are automatically at the poles of the latitude foliation.\newline
Following Section~\ref{sec:JSdif}, a hemispherical surface corresponds to a totally real Jenkins-Strebel differential. Zeroes and simple poles of the quadratic differential can only be the conical singularities with integer angle. Besides, since the other singularities are double poles, the sum of their orders should be even. Consequently, these (primitive) quadratic differentials are in $\mathcal{Q}(4k,-2^{2k})$ or $\mathcal{Q}(2k+1,2k-1,-2^{2k})$. In both cases, Theorem~\ref{thm:31} implies that in these strata, a configuration of uniformly equal residues cannot be realized. If the array of angles contains $2k$ equal non-integer angles $c$, the configuration $c,\dots,c$ cannot be realized. If there is no non-integer angle, then, the $2k$ double poles of the quadratic differentials correspond to regular points of the spherical metric and their quadratic residue is $1$. In this case, it cannot be realized either.
\end{proof}

Apart from these four exceptional families, the only condition that should be satisfied is the strengthened Gauss-Bonnet inequality stated in Theorem~\ref{thm:GBplus}.

\begin{thm}\label{thm:52}
Let $2\pi(a_{1},\dots,a_{n_{E}},b_{1},\dots,b_{n_{O}},c_{1},\dots,c_{n_{N}})$ be an array of angles as in Section~\ref{sec:GBplus}. Apart from the obstructions in genus one described in Proposition~\ref{prop:51}, there exists a cone spherical metric with strict dihedral monodromy on a surface of genus $g \geq 1$ with $n$ conical singularities of prescribed angles if and only if strengthened Gauss-Bonnet inequality is satisfied.
\end{thm}

\begin{proof}
We assume that the array of angles satisfies the strengthened Gauss-Bonnet inequality (Theorem~\ref{thm:GBplus}) and is not affected by the obstructions of Proposition~\ref{prop:51}. We are going to prove the existence of a cone spherical metric with strict dihedral monodromy in the category of hemispherical surfaces. For this purpose, we first need to select the stratum where we have to find to adequate quadratic differential.\newline
If $n_{O}$ is even, we introduce the orders $d_{1},\dots,d_{n_{E}+n_{O}}$ where:
\begin{itemize}
    \item $d_{i} = 2a_{i}-2$ if $1 \leq i \leq n_{E}$;
    \item $d_{i}=2b_{i-n_{E}}-2$ if $n_{E}+1 \leq i \leq n_{E}+n_{O}$.
\end{itemize}
Sum $d_{1}+\dots+d_{s}=2T-2s$ is even (here, $s=n_{E}+n_{O}$). Since the strengthened Gauss-Bonnet inequality is satisfied, we have $T \geq 2g+n-1$ if $n_{N}=0$ and $T \geq 2g+n-2$ otherwise. In other words, we have:
\begin{itemize}
    \item $d_{1}+\dots+d_{s} \geq 4g-2$ if $n_{N}=0$;
    \item $d_{1}+\dots+d_{s} \geq 4g+2n_{N}-4$ otherwise.
\end{itemize}
Consequently, in both cases, there is a number $t \geq 1$ such that $d_{1}+\dots+d_{s}=4g-4+2t$ and $t \geq n_{N}$. We will work with nonempty stratum $\mathcal{Q}=\mathcal{Q}(d_{1},\dots,d_{s},-2^{t})$ of primitive quadratic differentials on Riemann surfaces of genus $g$.\newline
The reasoning is essentially the same if $n_{O}$ is odd. We just have to take the $N_{O}-1$ bigger numbers among $b_{1},\dots,b_{n_{O}}$.\newline

Now, we need to find a totally real Jenkins-Strebel differential in stratum $\mathcal{Q}$ such that the quadratic residues at the double poles are:
\begin{itemize}
    \item $(c_{1})^{2},\dots,(c_{n_{N}})^{2},1,\dots,1$ if $n_{O}$ is even;
    \item $(b_{1})^{2},(c_{1})^{2},\dots,(c_{n_{N}})^{2},1,\dots,1$ if $n_{O}$ is odd (and the $b_{1},\dots,b_{n_{O}}$ are in increasing order).
\end{itemize}
Such a quadratic differential will correspond to a hemispherial surface where conical singularities in the equatorial net are of angles $2d_{i}+2$ while conical singularities in the polar locus are of angles $2\pi(r_{1},\dots,r_{t})$ where the quadratic residues are $(r_{1})^{2},\dots,(r_{t})^{2}$.\newline
If $g \geq 2$, any configuration of real positive quadratic residues can be realized in the stratum (Theorem~\ref{thm:31}) and they are also realized by totally real Jenkins-Strebel differential (see Section~\ref{sec:flotcontr}). Therefore, we can construct the analogous hemispherical surface (see Section~\ref{sec:JSdif}).\newline
If $g=1$, unless the maximal integral sum (see Definition~\ref{def:41}) is realized by angles $(4k+2)\pi$ or $(2k+3)\pi,(2k+1)\pi$, the same construction can be carried out (indeed, in Theorem~\ref{thm:31}, the only strata with obstructions are $\mathcal{Q}(4k,-2^{2k})$ and $\mathcal{Q}(2k+1,2k-1,-2^{2k})$). In the remaining cases, the only configurations of residues that cannot be realized are formed by an even number of identical residues. Consequently, the remaining conical singularities of the spherical structure that correspond to double poles of the quadratic differentials should have equal angles. This rules out the possibility of odd singularities among them. There are thus two possibilities. Either every double pole has a quadratic residue equal to $1$ (because it corresponds to a regular point of the spherical metric) or they all correspond to non-integer singularities. These cases are covered by Proposition~\ref{prop:51}.
\end{proof}

\subsection{Co-axial case}
\label{sec:co-axial}

In genus at least one, there is no obstruction to the existence of Abelian differentials with prescribed singularities (apart from the residue theorem).\newline

The following result is analogous to Theorem~1 of \cite{Er} concerning spherical metrics with co-axial monodromy on punctured spheres.

\begin{thm}\label{thm:53}
Let $2\pi(a_{1},\dots,a_{n},c_{1},\dots,c_{p})$ be an array of angles where $a_{1},\dots,a_{n} \in \mathbb{N}^{\ast}$ and $c_{1},\dots,c_{p} \notin \mathbb{N}^{\ast}$. There exists a spherical metric with co-axial monodromy on a surface of genus $g \geq 1$ with $n+p$ conical singularities of prescribed angles if and only if:
\begin{enumerate}
\item there is a sequence of signs $\epsilon_{1},\dots,\epsilon_{p}$ such that $K=\sum\limits_{j=1}^{p} \epsilon_{j}c_{j}\in \mathbb{N}$;
\item $\sum a_{i} -2g+2-n-p-K$ is nonnegative and even if $p \geq 1$;
\item $\sum a_{i} -2g+2-n$ is positive and even if $p=0$.
\end{enumerate}
\end{thm}

\begin{proof}
An array of angles is realized by a metric with co-axial monodromy if and only if it is realized by some hemispherical surface (see Proposition~\ref{prop:isomonodromy}). The totally real Jenkins-Strebel differential corresponding to such a hemispherical surface is the square of an Abelian differential (see Sections~\ref{sec:quaddifhalf} and~\ref{sec:JSdif}).\newline

We first prove that the existence of a cone spherical metric implies that the three conditions are satisfied. Such a metric is defined by a meromorphic $1$-form $\omega$ (see Section~\ref{sec:JSdif}). The zeroes of $\omega$ correspond to the conical singularities belonging to the equatorial net. Without loss of generality, we assume that their angles are the $s$ first numbers of the array $2\pi(a_{1},\dots,a_{n})$. The differential $\omega$ belongs to stratum $\mathcal{H}(a_{1}-1,\dots,a_{s}-1,-1^{t})$ where $t=2-2g+\sum\limits_{i=1}^{s} (a_{i}-1)$.\newline
Each of the $t$ simple poles of $\omega$ can be identified with a point in the cone spherical metric:
\begin{itemize}
    \item $n-s$ conical singularities of angles $2\pi(a_{s+1},\dots,a_{n})$ are simple poles of $\omega$ with residues equal to $\pm a_{i}$;
    \item $p$ conical singularities of angles $2\pi(c_{1},\dots,c_{p})$ are simple poles of $\omega$ with residues equal to $\pm c_{i}$;
    \item the remaining $t+s-p-n$ simple poles of $\omega$ have residues equal to $\pm 1$ and they correspond to regular points of the spherical metric.
\end{itemize}
We denote by $K$ the sum of the residues corresponding to the non-integer conical singularities of the spherical metric. The residue theorem implies that $K$ is an integer number. Since the $1$-form is defined up to a global change of sign, we can decide that $K$ is nonnegative. Condition (1) is proved.\newline
Since the sum of the residues of $\omega$ is equal to zero, the parity of $K$ is equal to the parity of $\sum\limits_{i=s+1}^{n}a_{i}+t+s-p-n$. The parity of $t$ is the same as the parity of $s+\sum\limits_{i=s+1}^{n}a_{i}$. It follows that the parity of $K$ is equal to the parity of $\sum\limits_{i=1}^{n} a_{i} +n+p$. We just proved that $L=\sum a_{i} -2g+2-n-p-K$ is even. It remains to prove that $L$ is nonnegative if $p \geq 1$ and positive if $p=0$.\newline
In the case $p=0$, the proof is as follows. The differential $\omega$ has at least one pole (because it is impossible for a holomorphic $1$-form to have only real periods). The inequality $t \geq 1$ implies that $2-2g-s+\sum\limits_{i=1}^{s} a_{i} \geq 1$. It follows that $L \geq 1+\sum\limits_{i=s+1}^{n} (a_{i}-1) \geq 1$.\newline
In the case $p \geq 1$, the number of double poles with residues equal to $\pm 1$ is given by $t-p=2-2g-p+\sum\limits_{i=1}^{s} (a_{i}-1)$. Since the sum of the residues is equal to zero, we have $t-p \geq K$. It follows that $2-2g-p-K+\sum\limits_{i=1}^{s} (a_{i}-1) \geq 0$. Finally, we obtain $L \geq \sum\limits_{i=s+1}^{n} (a_{i}-1) \geq 0$.\newline

Conversely, we prove that the three conditions implies the existence of a cone spherical metric with co-axial monodromy. We consider Abelian differentials in stratum $\mathcal{H}=\mathcal{H}(a_{1}-1,\dots,a_{n}-1,-1^{t})$ where $t=\sum a_{i} -n+2-2g$. If $p=0$, then the third condition implies that $t \geq 2$. If $p \geq 1$, then the second condition implies that $t \geq K +p$. Besides, if $p=1$, then $K \neq 0$ so $t \geq 2$ in any case. Since the only empty strata are those in which the unique pole is a simple pole, $\mathcal{H}$ is nonempty.\newline
Following Theorem~\ref{thm:33}, every array of residues summing to zero is realized in a nonempty stratum parameterizing $1$-forms on surfaces of genus at least one. Consequently, it remains to find an array of residues of the form $\pm c_{1}, \dots, \pm c_{p},\pm 1, \dots, \pm 1$ and summing to zero.\newline
The first condition proves the existence of a family of signs $\epsilon_{1},\dots,\epsilon_{p}$ such that $\sum\limits_{j=1}^{p} \epsilon_{j}c_{j} = K \in \mathbb{N}$. The remaining $t-p$ residues are of the form $\pm 1$. The second and third conditions prove that $t-p \geq K$. Besides $t-p$ and $K$ have the same parity. Consequently, we just have to pick $\frac{t-p+K}{2}$ residues equal to $-1$ and $\frac{t-p-K}{2}$ residues equal to $1$. Finally, we obtain the desired array of residues realized by a $1$-form in the desired stratum. The cone spherical metric defined by the $1$-form has co-axial monodromy and its conical singularities realize the desired array of angles.
\end{proof}

\subsection{Comparison}
\label{sec:comparison}

In genus at least one, almost every array of angles that are realized with co-axial monodromy can also be realized with strict dihedral monodromy. The exceptions are two of the four families of Proposition~\ref{prop:51}.

\begin{prop}\label{prop:54}
If an array of angles $2\pi(a_{1},\dots,a_{n},c_{1},\dots,c_{p})$ is realized by a spherical metric with co-axial monodromy on a surface of genus $g \geq 1$, then it can also be realized by a metric with strict dihedral monodromy on a surface of genus $g$ with the following two families of exceptions in genus one (parametrized by $k \in \mathbb{N}^{\ast}$):
\begin{itemize}
\item $(4k+2)\pi,c,\dots,c$ with $2k$ angles $c \notin \pi\mathbb{Z}$;
\item $(4k+2)\pi$.
\end{itemize}
These two families present obstruction for dihedral monodromy (Proposition~\ref{prop:51}) but clearly satisfy the hypotheses of Theorem~\ref{thm:53}.
\end{prop}

\begin{proof}
Since the array is realized by a metric with co-axial monodromy, the sum of integer angles satisfies $\sum a_{i} \geq 2g-2+n+p$ (see Theorem~\ref{thm:53}). If $g \geq 2$ and $p \geq 1$, this implies the existence of a metric with strict dihedral monodromy (Theorem~\ref{thm:52}).\newline
If $g \geq 2$ but $p=0$, classical Gauss-Bonnet implies $\sum a_{i} > 2g-2+n$ and thus $\sum a_{i} \geq 2g+n-1$. This also implies the existence of a metric as required.\newline
Then we consider the genus one case. If $p=0$, then we have $\sum a_{i} > n$ (see Theorem~\ref{thm:53}). Theorem~\ref{thm:52} then implies the existence of metric with strict dihedral monodromy except in the four exceptional cases of Proposition~\ref{prop:51}. Among them, the only family for which there is obstruction is where there is only one conical singularity the angle of which is of the form $(4k+2)\pi$.\newline
Then, we assume $g=1$ and $p \geq 1$. We have $\sum a_{i} \geq 2g-2+n+p$. Therefore, the strengthened Gauss-Bonnet inequality is satisfied and Theorem~\ref{thm:52} implies the existence of a metric with strict dihedral monodromy except in the cases of Proposition~\ref{prop:51}. Since there is at least one even singularity and $p \geq 1$, we just have to avoid the case where the array of angles is $(4k+2)\pi,c,\dots,c$ with $2k$ non-integer angles $c \notin \pi\mathbb{Z}$.
\end{proof}

\begin{rem}
The special case of cone spherical metrics on a torus with only one conical singularity is already treated in~\cite{EMP}.
\end{rem}

\section{Spherical structures in genus zero}
\label{sec:genuszero}

The monodromy of cone spherical metrics on punctured spheres is generated by rotations around the singularities. In particular, if every singularity is even (its angle belongs to $2\pi\mathbb{Z}$), then the monodromy of the metric is trivial and the geometric structure is just a cover of the sphere ramified at these singularities. This case has been already classified, see Section~7 in \cite{Er1}.\newline

In the following, we assume at least one conical singularity has a non-integer angle. As we will see in Section~\ref{sec:co-axialgzero}, this automatically implies the existence of a second conical singularity with a non-integer angle.\newline

\subsection{Co-axial case}
\label{sec:co-axialgzero}

In this section, we consider geometric structures whose monodromy group is contained in the group of rotations around an axis. In particular, it could be a finite rotation group (if the angles of the non-integer singularities belong to $\pi\mathbb{Q}$ for example).

Theorem~1 in \cite{Er} gave a complete classification of the arrays of angles that are realized by a spherical metric with a co-axial monodromy.

\begin{thm}[\cite{Er}]\label{thm:61}
Let $2\pi(a_{1},\dots,a_{n},c_{1},\dots,c_{p})$ be an array of angles where $a_{1},\dots,a_{n} \in \mathbb{N}^{\ast}$, $c_{1},\dots,c_{p} \notin \mathbb{N}^{\ast}$ and $p \geq 1$. There exists a spherical metric with co-axial monodromy on a punctured sphere with $n+p$ conical singularities of prescribed angles if and only if:
\begin{itemize}
\item there is a sequence of signs $\epsilon_{1},\dots,\epsilon_{p}$ such that $K=\sum_{j=1}^{p} \epsilon_{j}c_{j}\in \mathbb{N}$;
\item $M=\sum a_{i} +2-n-p-K$ is nonnegative and even;
\item if $c_{1},\dots,c_{p}$ are commensurable, an additional arithmetic condition should hold.
\end{itemize}
Let $v$ be the vector $(c_{1},\dots,c_{p},1,\dots,1)$ with $M+K$ elements equal to $1$. If $v$ is of the form $L(b_{1},\dots,b_{p+M+K})$ with $L>0$ and $b_{1},\dots,b_{p+M+K}$ are integers, then the additional condition is:
$$2\max(a_{1},\dots,a_{n}) \leq \sum b_{i}.$$
\end{thm}

The condition about signed sums of non-integer angles means in particular that there should be at least two non-integer singularities.\newline

A proof analogous to the one of Theorem~\ref{thm:53} could be given to Theorem~\ref{thm:61}  by using Theorem~\ref{thm:34}.\newline

\subsection{Strict dihedral case}
\label{sec:stricdihedralgzero}

In the problem of the realization of an array of angles by a metric with strict dihedral monodromy, the nature of the obstructions depends crucially on the number of odd singularities.

We first give a restriction on the number of odd singularities.

\begin{lem}\label{lem:62}
Let $2\pi(a_{1},\dots,a_{n_{E}},b_{1},\dots,b_{n_{O}},c_{1},\dots,c_{n_{N}})$ be an array of angles as in Section~\ref{sec:GBplus}. If it is realized by a cone spherical metric with strict dihedral monodromy on the punctured sphere, then  $n_{O} \geq 2$.
\end{lem}

\begin{proof}
If the array is realized, then it can be realized by a hemispherical surface (Proposition~\ref{prop:isomonodromy}). The latter corresponds to a primitive quadratic differential (Section~\ref{sec:quaddifhalf}). In genus zero, primitive quadratic differentials have singularities of odd order (otherwise they are global squares of $1$-forms). Besides, they have an even number of singularities of odd order (since the total order is $-4$). Consequently, there are at least two odd singularities belonging to the equatorial net of the hemispherical surface.
\end{proof}

We now treat the cases according to the number of odd singularities. We begin with the case with at least~$4$ odd singularities, then we treat the case with~$3$ odd singularities and finally with~$2$.

\subsubsection{At least four odd singularities}
\label{sec:quatresing}

In the case $n_{O} \geq 4$, there is no additional obstruction to the strengthened Gauss-Bonnet inequality.

\begin{thm}\label{thm:63}
Let $2\pi(a_{1},\dots,a_{n_{E}},b_{1},\dots,b_{n_{O}},c_{1},\dots,c_{n_{N}})$ be an array of angles as in Section~\ref{sec:GBplus}. If $n_{O} \geq 4$, there exists a cone spherical metric with strict dihedral monodromy on a punctured sphere with $n$ conical singularities of prescribed angles if and only if the strengthened Gauss-Bonnet inequality is satisfied. Besides, the metric can be chosen in such a way its monodromy group is infinite.
\end{thm}

\begin{proof}
Let assume that the array of angles satisfies the strengthened Gauss-Bonnet inequality (Theorem~\ref{thm:GBplus}). We prove the existence of a cone spherical metric with strict dihedral monodromy in the category of hemispherical surfaces.\newline
Since the strengthened Gauss-Bonnet inequality is satisfied, we can find in the array of angles a subset of $s$ even or odd singularities of angles $\pi(2+d_{i})$ such that $d_{1}+\dots+d_{s}$ is even number whose value is at least $4g-4+2(n-s)$ (or strictly bigger if $n_{O}$ is even and $n_{N}=0$). Besides, we assume there are at least four odd singularities among the $s$ chosen singularities. Then, we consider a nonempty stratum $\mathcal{Q}(d_{1},\dots,d_{s},-2^{t})$ of primitive quadratic differentials on the Riemann sphere.  Each quadratic differential  of these strata have at least four singularities of odd order. Therefore, any array of real positive quadratic residues is realized in the stratum (Theorem~\ref{thm:32}). Besides, any array is also realized by totally real Jenkins-Strebel differential (see Section~\ref{sec:flotcontr}). Therefore, we can construct the analogous hemispherical surface (see Section~\ref{sec:JSdif}).\newline
Finally, we prove that we can realize the hemispherical surface in such a way the projection of the monodromy group on the rotation group around the preserved axis has dense image. For a quadratic differential on the Riemann sphere with at least four odd singularities, the canonical double cover ramified at the odd singularities is of genus at least one (by Riemann-Hurwitz formula). Therefore, there are other degrees of freedom than the quadratic residues. Up to a small perturbation, we can assume the group of absolute periods of the Abelian differential in the cover is dense in $\mathbb{R}$.
\end{proof}

\subsubsection{Three odd singularities}
\label{sec:troisimp}

If $n_{O}=3$, there are specific obstructions that require to be handled separately.

\begin{prop}\label{prop:64}
Let $k \in \mathbb{N}$, $l \geq k$ and $\alpha,\beta \notin \pi\mathbb{Z}$. If 
$$\alpha+\beta=(2l+1)\pi \text{ or } \alpha+(2l+1)\pi=\beta \text{ or }\beta+(2l+1)\pi=\alpha\, ,$$ then the array $( (2k+1)\pi,(2k+1)\pi,(2l+1)\pi,\alpha,\dots,\alpha,\beta)$ with $2k-1$ angles $\alpha$ cannot be realized by a cone spherical metric with strict dihedral monodromy on a punctured sphere.
\end{prop}

\begin{proof}
If such an array were realized by a cone spherical metric, then it would be also realized by a hemispherical surface. The latter should have at least two half-integer singularities on its equatorial net (see Lemma~\ref{lem:62}). This implies the existence of a quadratic differential in $\mathcal{Q}(2k-1,2k-1,-2^{2k+1})$ whose quadratic residues are $\left(l+\frac{1}{2}\right)^{2},\left(\frac{\beta}{2\pi}\right)^{2}$ and $2k-1$ quadratic residues equal to $\left(\frac{\alpha}{2\pi}\right)^{2}$. This configuration is forbidden by Theorem~\ref{thm:32}.
\end{proof}

In addition to strengthened Gauss-Bonnet inequality stated in Theorem~\ref{thm:GBplus} and the obstruction of Proposition~\ref{prop:64}, an arithmetic condition should be satisfied if every conical singularity has a rational angle.

\begin{thm}\label{thm:65}
Let $2\pi(a_{1},\dots,a_{n_{E}},b_{1},b_{2},b_{3},c_{1},\dots,c_{n_{N}})$ be an array of angles. We assume that $b_{1} \geq b_{2} \geq b_{3}$. Apart from the obstructions described in Proposition~\ref{prop:64}, there exists a cone spherical metric with strict dihedral monodromy on a punctured sphere with $n$ conical singularities of prescribed angles if and only if:
\begin{itemize}
\item the strengthened Gauss-Bonnet inequality $T=\sum a_{i} +b_{1}+b_{2} \geq n-2$ holds;
\item an additional arithmetic condition described below is satisfied when $c_{1},\dots,c_{n_{N}}\in \pi\mathbb{Q}$.
\end{itemize}
If vector $(c_{1},\dots,c_{n_{N}},b_{3},1,\dots,1)$ with $T+2-n$ elements equal to $1$ is of the form $L(r_{1},\dots,r_{T-n_{E}})$ with
$L>0$ and $r_{1},\dots,r_{T-n_{E}}$ integers, then:
\begin{itemize}
\item $\sum r_{i} \geq 2b_{1}+2b_{2}$ if $\sum r_{i}$ is even;
\item $\sum r_{i} \geq 2b_{1}$ if $\sum r_{i}$ is odd.
\end{itemize}
\end{thm}

\begin{proof}
We first prove that the existence of a cone spherical metric implies that the conditions on the array of angles are satisfied. Theorem~\ref{thm:GBplus} and Proposition~\ref{prop:64} already prove that the array should satisfy the strengthened Gauss-Bonnet inequality and be unaffected by the exceptional obstructions described in Proposition~\ref{prop:64}. It remains to prove that in the case where $c_{1},\dots,c_{n_{N}}\in \pi\mathbb{Q}$, the additional arithmetic condition is satisfied.\newline
Following Proposition~\ref{prop:isomonodromy}, the cone spherical metric is defined by a totally real Jenkins-Strebel primitive quadratic differential $\omega$. Since $\omega$ is a primitive quadratic differential on a sphere, exactly two of its zeroes and poles are of odd order. Consequently, there is a number $s$ satisfying $0 \leq s \leq n_{E}$ and a permutation $\phi \in \mathfrak{S}_{3}$ such that $\omega$ belongs to the stratum $\mathcal{Q}(2a_{1}-2,\dots,2a_{s}-2,2b_{\phi(1)}-2,2b_{\phi_{2}}-2,-2^{t})$ where $2t=2b_{\phi(1)}+2b_{\phi(2)}-2s+2\sum\limits_{i=1}^{s} a_{i}$.\newline
The configuration of quadratic residues realized by $\omega$ is:
$$(c_{1})^{2},\dots,(c_{n_{N}})^{2},(b_{\phi(3)})^{2},(a_{s+1})^{2},\dots,(a_{n_{E}})^{2},1,\dots,1.$$
By hypothesis, the configuration is of the form $M\cdot\left((x_{1})^{2},\dots,(x_{t})^{2} \right)$ where $M>0$ and $x_{1},\dots,x_{t}$ are coprime integer numbers. Theorem~\ref{thm:32} then implies that:
\begin{itemize}
    \item $\sum x_{i} > 2b_{\phi(1)}+2b_{\phi(2)}-2$ if $\sum x_{i}$ is even;
    \item $\sum x_{i} > 2\max(b_{\phi(1)},b_{\phi(2)})-2$ if $\sum x_{i}$ is odd.
\end{itemize}
We first consider the case $t>n_{N}+1+(n_{E}-s)$. In other words, at least one of the quadratic residues is equal to $1$. Since $b_{\phi(3)} \in \mathbb{Z}+\frac{1}{2}$, $\sqrt{M}$ is the lowest common denominator of $c_{1},\dots,c_{p},\frac{1}{2}$. It follows that $M=L$. Besides note that 
$$\sqrt{M}\sum x_{i}= \sum\limits_{i=s+1}^{n_{E}}a_{i} + b_{\phi(3)} + \sum c_{i} + (t-n_{N}-1-(n_{E}-s))\,.$$
We deduce that 
$$\sqrt{M}\sum x_{i}= \sum a_{i} +\sum b_{i} + \sum c_{i} -n_{E}-n_{E}-1=b_{3}+ \sum c_{i} + T +2-n = \sqrt{L}\sum r_{i}\,.$$ 
Since $2b_{1}+2b_{2} \geq 2b_{\phi(1)}+2b_{\phi(2)}$, $2b_{1} \geq 2\max(b_{\phi(1)},b_{\phi(2)})$ and $b_{3} \leq b_{\phi(3)}$, it follows that the vector $(c_{1},\dots,c_{n_{N}},b_{3},1,\dots,1)$ satisfies the arithmetic condition (the parity of $\sum x_{i}$ does not depend on the choice of $s$ and $\phi$).\newline
If $t=n_{N}+1+(n_{E}-s)$, then we have $b_{\phi(1)}+b_{\phi(2)}-s+\sum\limits_{i=1}^{s} a_{i}=n_{N}+1+(n_{E}-s)$. After simplification, we obtain $T=n_{N}+1+n_{E}+\sum\limits_{i=s+1}^{n_{E}}a_{i}+(b_{1}+b_{2}-b_{\phi(1)}-b_{\phi(2)})$. The strengthened Gauss-Bonnet inequality $T \geq n-2$ then implies that $s=n_{E}$ and $b_{1}+b_{2}=b_{\phi(1)}+b_{\phi(2)}$. It follows that $b_{3}=b_{\phi(3)}$. The vector $(c_{1},\dots,c_{n_{N}},b_{3},1,\dots,1)$ thus coincides with the configuration of residues realized by $\omega$ and satisfies the arithmetic condition.\newline

Conversely, we prove that any array satisfying the stated conditions is realized by a cone spherical metric with strict dihedral monodromy. We construct a totally real Jenkins-Strebel primitive quadratic differential corresponding a hemispherical surface realizing the array of angles (see Sections~\ref{sec:quaddifhalf} and~\ref{sec:JSdif}).\newline
We consider the stratum $\mathcal{Q} = \mathcal{Q}(2a_{1}-2,\dots,2a_{n_{E}}-2,2b_{1}-2,2b_{2}-2,-2^{t})$. Since the array of angles satisfied the strengthened Gauss-Bonnet inequality, we have $t = \sum a_{i} +b_{1}+b_{2} -n_{E} = T-n_{E} \geq 1$. It follows that $\mathcal{Q}$ is nonempty. It remains to prove that the configuration of residues
$(c_{1},\dots,c_{n_{N}},b_{3},1,\dots,1)$ with $T+2-n$ elements equal to $1$ is realized in $\mathcal{Q}$.\newline
We just need to check that none of the four obstructions described in Theorem~\ref{thm:32} applies to $(c_{1},\dots,c_{n_{N}},b_{3},1,\dots,1)$. The additional arithmetic conditions exactly imply that the third and fourth obstructions do not apply. Since $b_{3}$ is not equal to any other number of the list, the second obstruction of Theorem~\ref{thm:32} is not relevant.\newline
If the first obstruction of Theorem~\ref{thm:32} would apply, then $n_{E}=0$ and $b_{1}=b_{2}$. The array of residues contains three numbers $A^{2},B^{2},C^{2}$ satisfying the equation $A=B+C$. If one of these numbers is equal to $1$, the fact that $b_{3} \in \mathbb{Z}+\frac{1}{2}$ implies that $A,B,C \in \frac{1}{2}\mathbb{Z}$. This implies $n_{N}=0$ and the array of angles is thus $2b_{1}\pi,2b_{1}\pi,2b_{3}\pi,b_{3}\pi,\dots,b_{3}\pi$ with $2b_{1}$ angles equal to $b_{3}\pi$. Proposition~\ref{prop:64} rules out this case.\newline
Now we assume $1$ is not one of the three numbers $A,B,C$. It follows that $t=n_{N}+1$. The array is then $2\pi(b_{1},b_{1},b_{3},c_{1},\dots,c_{1},c_{n_{N}})$ where $b_{3},c_{1},c_{n_{N}}$ satisfy an equality of the form $A=B+C$. These arrays of angles are also ruled out by Proposition~\ref{prop:64}. Consequently, no obstruction of Theorem~\ref{thm:32} applies and the array of angles is realized by a hemispherical surface corresponding to a quadratic differential of $\mathcal{Q}$.\newline
Finally, it should be noted that the obtained hemispherical surface has strict dihedral monodromy. Indeed, there are odd singularities on the equatorial net as well as on the poles of the latitude foliation. Therefore, the monodromy is not co-axial.
\end{proof}

\begin{rem}
In the arrays that are forbidden by the additional arithmetic condition, there should always be two non-integer singularities with equal angles (at least two elements among $r_{1},\dots,r_{n_{N}+K+1}$ should be equal to one). They have smaller angle than any other singularity.
\end{rem}

\begin{ex}
The array $3\pi,3\pi,3\pi,\frac{3\pi}{2},\frac{3\pi}{2}$ cannot be realized by a metric with dihedral monodromy because of the additional arithmetic obstruction. Indeed in that case the array $(b_{3},c_{1},c_{2})$ is of the form $L(2,1,1)$ with $L=\left(\frac{3}{4}\right)^{2}$. Hence the first part of the arithmetic obstruction of Theorem~\ref{thm:65} applies.
\end{ex}

\subsubsection{Two odd singularities}
\label{sec:deuxodd}

If $n_{O}=2$, then it has already been proved in Section~4 of~\cite{EGT} that in absence of non-integer singularities, the monodromy of the cone spherical metric is co-axial. Therefore we will assume $n_{N} \geq 1$.\newline

As previously, some specific obstructions have to be handled separately. They correspond to the first two obstructions of Theorem~\ref{thm:32}.

\begin{prop}\label{prop:68}
For $k \in \mathbb{N}$, $\alpha,\beta \notin \pi\mathbb{Z}$, the following arrays of angles are not realized by a cone spherical metric with strict dihedral monodromy on a punctured sphere:
\begin{itemize}
\item $((2k+3)\pi,(2k+1)\pi,\alpha,\dots,\alpha,\beta,\beta)$ with $2k$ angles $\alpha$;
\item $((2k+3)\pi,(2k+1)\pi,\alpha,\dots,\alpha)$ with $2k$ angles $\alpha$;
\item $((2k+3)\pi,(2k+1)\pi,\alpha,\alpha)$.
\end{itemize}
\end{prop}

\begin{proof}
If such an array is realized by a cone spherical metric, then it is also realized by a hemispherical surface. The latter should have at least two half-integer singularities on its equatorial net. This implies the existence of a quadratic differential in $\mathcal{Q}(2k+1,2k-1,-2^{2k+2})$ where zeroes are the two odd singularities of the spherical metric while the double poles correspond either to non-integer singularities or regular points. Therefore, the  quadratic residues can be equal to $\left(\frac{\alpha}{2\pi}\right)^{2}$ or $\left(\frac{\beta}{2\pi}\right)^{2}$ (if the double pole corresponds to a non-integer singularity) or are equal to $1$ (if the double pole corresponds to a regular point of the spherical metric).\newline
In any case, the second obstruction of Theorem~\ref{thm:32} forbids the existence of a quadratic differential with such a configuration of quadratic residues.
\end{proof}

\begin{prop}\label{prop:69}
For $k \in \mathbb{N}$, $\alpha,\beta,\gamma \notin \pi\mathbb{Z}$, the following arrays of angles are not realized by a cone spherical metric with strict dihedral monodromy on a punctured sphere:
\begin{itemize}
\item $((2k+3)\pi,(2k+3)\pi,\alpha,\dots,\alpha,\beta,\gamma)$ with $2k+1$ angles equal to $\alpha$ and $\alpha=\beta+\gamma$;
\item $((2k+3)\pi,(2k+3)\pi,\alpha,\dots,\alpha,\beta,\alpha+\beta)$ with $2k+1$ angles $\alpha$;
\item $((2k+3)\pi,(2k+3)\pi,\alpha,\dots,\alpha,\beta)$ with $2k+1$ angles equal to $\alpha$, $\alpha+\beta=2\pi$, $\beta=\alpha+2\pi$ or $\alpha=\beta+2\pi$;
\item $((2k+3)\pi,(2k+3)\pi,\alpha,\alpha+2\pi)$;
\item $((2k+3)\pi,(2k+3)\pi,\alpha,\beta)$ with $\alpha+\beta=2\pi$.
\end{itemize}
\end{prop}

\begin{proof}
We proceed in the same way as in the proof of Proposition~\ref{prop:68} except that in this case we refer to the first obstruction of Theorem~\ref{thm:32}.
\end{proof}

In addition to strengthened Gauss-Bonnet inequality (see Theorem~\ref{thm:GBplus}) and the obstructions of Proposition~\ref{prop:68} and~\ref{prop:69}, an arithmetic condition should be satisfied if the non-integer singularities have commensurable angles.

\begin{thm}\label{thm:610}
Let $2\pi(a_{1},\dots,a_{n_{E}},b_{1},b_{2},c_{1},\dots,c_{n_{N}})$ be an array of angles such that $n_{N} \geq 1$ and $b_{1} \geq b_{2}$. Apart from the obstructions described in Proposition~\ref{prop:68} and~\ref{prop:69}, there exists a cone spherical metric with strict dihedral monodromy on a punctured sphere with $n$ conical singularities of prescribed angles if and only if the two following conditions hold:
\begin{itemize}
\item the strengthened Gauss-Bonnet inequality $T=\sum a_{i} +b_{1}+b_{2} \geq n-2$ holds;
\item an additional arithmetic condition described below is satisfied when $c_{1},\dots,c_{n_{N}}$ are commensurable.
\end{itemize}
Considering the vector $v=(c_{1},\dots,c_{n_{N}},1,\dots,1)$ with $T+2-n$ elements equal to $1$, if $v$ is of the form $L(r_{1},\dots,r_{T-n_{E}})$ with $L>0$ and $r_{1},\dots,r_{T-n_{E}}$ integers, then:
\begin{itemize}
\item $\sum r_{i} \geq 2b_{1}+2b_{2}$, if $\sum r_{i}$ is even;
\item $\sum r_{i} \geq 2b_{1}$, if $\sum r_{i}$ is odd.
\end{itemize}
\end{thm}

\begin{proof}
The steps of the proof are essentially the same as for Theorem~\ref{thm:65}. We first have to prove that the existence of a cone spherical metric implies that the conditions on the array of angles are satisfied. Using Theorem~\ref{thm:GBplus}, Proposition~\ref{prop:68} and Proposition~\ref{prop:69}, it remains just to prove that in the case where $c_{1},\dots,c_{n_{N}}$ are commensurable, the additional arithmetic condition is satisfied.\newline
The cone spherical metric corresponds to a primitive quadratic differential $\omega$ belonging to a stratum $\mathcal{Q}(2a_{1}-2,\dots,2a_{s}-2,2b_{1}-2,2b_{2}-2,-2^{t})$ where:
\begin{itemize}
    \item the number $s$ satisfies $0 \leq s \leq n_{E}$;
    \item $2t=2b_{1}+2b_{2}-2s+2\sum\limits_{i=1}^{s} a_{i}$.
\end{itemize}
The configuration of quadratic residues realized by $\omega$ is:
$$(c_{1})^{2},\dots,(c_{n_{N}})^{2},(a_{s+1})^{2},\dots,(a_{n_{E}})^{2},1,\dots,1.$$
By hypothesis, the configuration is of the form $M\cdot\left((x_{1})^{2},\dots,(x_{t})^{2}\right)$ where $M>0$ and $x_{1},\dots,x_{t}$ are coprime integer numbers. Theorem~\ref{thm:32} then implies that:
\begin{itemize}
    \item $\sum x_{i} > 2b_{1}+2b_{2}-2$ if $\sum x_{i}$ is even;
    \item $\sum x_{i} > 2\max(b_{1},b_{2})-2$ if $\sum x_{i}$ is odd.
\end{itemize}
Now we consider two cases depending whether $t>n_{N}+(n_{E}-s)$ or not. In the first case, at least one of the quadratic residues is equal to $1$. The number of elements equal to $1$ in this list is $t-n_{N}-(n_{E}-s)=b_{1}+b_{2}+\sum\limits_{i=1}^{s} a_{i} +2-n$. Thus, if it is positive for some value of $s$, it is also positive for $s=n_{E}$. It follows that $c_{1},\dots,c_{n_{N}}$ belong to $\pi\mathbb{Q}$ because $c_{1},\dots,c_{n_{N}},1$ are commensurable (otherwise the arithmetic condition does not need to be satisfied). Consequently, $\sqrt{M}=\sqrt{L}$ is the lowest common denominator of this list of numbers.\newline
We have $$\sqrt{M}\sum x_{i}= \sum\limits_{i=s+1}^{n_{E}}a_{i} + \sum c_{i} + (t-n_{N}-(n_{E}-s))\,,$$ and deduce that $$\sqrt{M}\sum x_{i}= \sum a_{i} +\sum b_{i} + \sum c_{i} +2-n= \sum c_{i} + T +2-n = \sqrt{L}\sum r_{i}\,.$$ It follows that vector $(c_{1},\dots,c_{n_{N}},1,\dots,1)$ satisfies the arithmetic condition (the parity of $\sum x_{i}$ does not depend on the choice of $s$).\newline
  
If $t=n_{N}+(n_{E}-s)$, then we have $b_{1}+b_{2}-s+\sum\limits_{i=1}^{s} a_{i}=n_{N}+(n_{E}-s)$. After simplification, we obtain $T=n_{N}+n_{E}+\sum\limits_{i=s+1}^{n_{E}}a_{i}$. The strengthened Gauss-Bonnet inequality $T \geq n-2$ then implies that $s=n_{E}$. The vector $(c_{1},\dots,c_{n_{N}},1,\dots,1)$ thus coincides with the configuration of residues realized by $\omega$ and satisfies the arithmetic condition.\newline

Conversely (as in the proof of Theorem~\ref{thm:65}), for any array of angles satisfying the conditions, we construct a totally real Jenkins-Strebel primitive quadratic differential corresponding to a hemispherical surface realizing the array of angles. We just have to realize the adequate configuration of quadratic residues in the adequate stratum.\newline
We consider the stratum $\mathcal{Q} = \mathcal{Q}(2a_{1}-2,\dots,2a_{n_{E}}-2,2b_{1}-2,2b_{2}-2,-2^{t})$. Since the array of angles satisfies the strengthened Gauss-Bonnet inequality, we have $t = \sum a_{i} +b_{1}+b_{2} -n_{E} = T-n_{E} \geq 1$. It follows that $\mathcal{Q}$ is nonempty. It remains to prove that the configuration of residues
$(c_{1},\dots,c_{n_{N}},1,\dots,1)$ with $T+2-n$ elements equal to $1$ is realized in $\mathcal{Q}$. This amounts to check that none of the four obstructions described in Theorem~\ref{thm:32} applies to $(c_{1},\dots,c_{n_{N}},b_{3},1,\dots,1)$. The additional arithmetic conditions exactly imply that the third and fourth obstructions do not apply.\newline
If the second obstruction of Theorem~\ref{thm:32} applies, then $n_{E}=0$, $b_{1}=b_{2}+1$ and the array of residues is of the form $A^{2},A^{2},B^{2},\dots,B^{2}$ (with an even number of quadratic residues equal to $B^{2}$). The array of angles is thus of one the forms ruled out by Proposition~\ref{prop:68}.\newline
If the first obstruction of Theorem~\ref{thm:32} applies, then $n_{E}=0$ and $b_{1}=b_{2}$. The array of residues contains three numbers $A^{2},B^{2},C^{2}$ satisfying equation $A=B+C$. All the corresponding arrays of angles are ruled out by Proposition~\ref{prop:69} (we should note that if one number of $A,B,C$ is equal to $1$, the corresponding point in the cone spherical metric is a regular point). No obstruction of Theorem~\ref{thm:32} applies and the array of angles is realized by a hemispherical surface corresponding to a quadratic differential of $\mathcal{Q}$.\newline
Finally, the hemispherical surface we obtained cannot have co-axial monodromy since it has odd singularities on the equator and at least one non-integer singularity at the poles since $n_{N} \geq 1$.
\end{proof}

\subsection{Comparison}
\label{sec:comparisongzero}

Just like in Proposition~\ref{prop:54}, some arrays of angles that can be realized by a spherical metric with co-axial monodromy can also be realized by a spherical metric with strict dihedral monodromy. We give a complete characterization.

\begin{prop}
Let $2\pi(a_{1},\dots,a_{n_{E}},b_{1},\dots,b_{n_{O}},c_{1},\dots,c_{n_{N}})$ be an array of angles. If it is realized by a spherical metric with co-axial monodromy on a punctured sphere, then it can also be realized by a metric with strict dihedral monodromy on a punctured sphere if and only if $n_{O} \geq 2$ and $n_{O}+n_{N}\geq 3$.
\end{prop}

\begin{proof}
Lemma~\ref{lem:62} implies that $n_{O} \geq 2$ is a necessary condition for the existence of a spherical metric with dihedral monodromy on a punctured sphere. Besides, if the number of singularities with nontrivial monodromy is exactly two, then the monodromy of the metric is automatically co-axial. Therefore, $n_{O}+n_{N}\geq 3$ is a necessary condition. We will prove that these two conditions are also necessary.\newline
We consider an array of angles realized by a spherical metric with co-axial monodromy. We assume that it satisfies the two necessary conditions. Theorem~\ref{thm:61} implies in particular that $\sum a_{i} \geq n_{E}+n_{O}+n_{N}-2$. Since we have $n_{O}+n_{N}\geq 3$, we deduce $n_{E} \geq 1$. Therefore, the array of angles is not forbidden by Propositions~\ref{prop:64},~\ref{prop:68} and~\ref{prop:69}. Besides, the condition on the sum of orders of even singularities in Theorem~\ref{thm:61} clearly implies the strengthened Gauss-Bonnet inequality (see Theorem~\ref{thm:GBplus}). It remains to prove that the array of angles satisfies the hypotheses of Theorems~\ref{thm:63},~\ref{thm:65} and~\ref{thm:610}.\newline
If $n_{O} \geq 4$, then there is no additional condition to check and Theorem~\ref{thm:63} implies that the array of angles is realized by a spherical metric with strict dihedral monodromy.\newline

If $n_{O}=2$ or $n_{O}=3$, we assume that $b_{1} \geq b_{2} \geq b_{3}$. We already know that $\sum a_{i} \geq n-2$. We set $K=T+2-n$ and $T=\sum a_{i} +b_{1}+b_{2} \geq n-2$ (Strengthened Gauss-Bonnet inequality). Thus, $K \geq b_{1}+b_{2}$. Theorem~\ref{thm:65} and~\ref{thm:610} require an additional arithmetic condition.\newline
We first consider the case $n_{O}=3$. The condition of Theorem~\ref{thm:65} is the following. Let~$v$ be the vector $(c_{1},\dots,c_{n_{N}},b_{3},1,\dots,1)$ with $K$ elements equal to $1$. If the array of angles cannot be realized by a metric with strict dihedral monodromy, then~$v$ is of the form $L(r_{1},\dots,r_{n_{N}+K+1})$ with $L>0$ and $r_{1},\dots,r_{n_{N}+K+1}$ are integers. We should have:
\begin{itemize}
\item $\sum r_{i} \geq 2b_{1}+2b_{2}$ if $\sum r_{i}$ is even;
\item $\sum r_{i} \geq 2b_{1}$ if $\sum r_{i}$ is odd.
\end{itemize}
If the condition is not satisfied, then $L=1$ because otherwise the $K \geq b_{1}+b_{2}$ elements of~$v$ that are equal to $1$ would be enough to satisfy the bound. However, $b_{3}$ is not an integer, which leads to a contradiction since $\frac{b_{3}}{L}$ is required to be an integer.\newline
In the case $n_{O}=2$, the condition of Theorem~\ref{thm:610} is essentially the same. We have to consider the vector $(c_{1},\dots,c_{n_{N}},b_{3},1,\dots,1)$ with $K$ elements equal to $1$. Similarly, if the condition is not satisfied, then $L=1$. This implies that $c_{1},\dots,c_{n_{N}}$ are integers, which is a contradiction. We already know by hypothesis that $n_{N} \geq 1$. This ends the proof.
\end{proof}

\paragraph{\bf Acknowledgements.} The authors would like to thank Alexandre Eremenko and the anonymous referee for valuable remarks. The second author would also like to thank Boris Shapiro for introducing him to the field of spherical metrics.\newline

\nopagebreak
\vskip.5cm
\end{document}